\documentclass[]{amsart}

\usepackage{graphicx}
\usepackage[abs]{overpic}

\theoremstyle{plain}
\newtheorem{theorem}{Theorem}[section]

\newtheorem{proposition}[theorem]{Proposition}

\theoremstyle{example}
\newtheorem{question}[theorem]{Question}

\theoremstyle{definition}
\newtheorem{definition}[theorem]{Definition}
\newtheorem{example}[theorem]{Example}
\newtheorem{remark}[theorem]{Remark}

\graphicspath{{../figures/}}

\begin{document}

\title{A construction of smooth travel groupoids on finite graphs}
\author{Diogo Kendy Matsumoto \and Atsuhiko Mizusawa}
\thanks{The first author is partly supported by a Waseda University Grant for Special Research Projects (Project number: 2014A-6133)}
\thanks{The second author is partly supported by a Waseda University Grant for Special Research Projects (Project number: 2014K-6134) and JSPS KAKENHI Grant Number 25287014}

\address{Department of Mathematics, Fundamental Science and Engineering, Waseda University, 3-4-1 Okubo, Shinjuku-ku, Tokyo 169-8555, Japan}
\email[Diogo Kendy Matsumoto]{diogo-swm@aoni.waseda.jp}
\email[Atsuhiko Mizusawa]{a\_mizusawa@aoni.waseda.jp}

%\institute{Diogo Kendy Matsumoto 
%\at Department of Mathematics, Fundamental Science and Engineering, Waseda University, 3-4-1 Okubo, Shinjuku-ku, Tokyo 169-8555, Japan\\
%\email{diogo-swm@aoni.waseda.jp}
%\and
%Atsuhiko Mizusawa 
%\at Department of Mathematics, Fundamental Science and Engineering, Waseda University, 3-4-1 Okubo, Shinjuku-ku, Tokyo 169-8555, Japan\\
%\email{a\_mizusawa@aoni.waseda.jp}}

%\subjclass
%05C05, 05C38, 20N02
%\endsubjclass

\maketitle

\begin{abstract}
%In 2006, L.~Nebesk$\acute{\mbox{y}}$ introduced an algebraic system called a travel groupoid.
A travel groupoid is an algebraic system related with graphs.
In this paper, we give an algorithm to construct smooth travel groupoids for any finite graph.
This algorithm gives an answer of L.~Nebesk$\acute{\mbox{y}}$'s question, 
``Does there exists a connected graph $G$ such that $G$ has no smooth travel groupoid?", in finite cases.
\end{abstract}
\section{Introduction}
Through the study of the algebraic characterization of geodetic graphs \cite{Neb98,Neb02} and trees \cite{Neb00}, 
L.~Nebesk$\acute{\mbox{y}}$ introduced an algebraic system called a travel groupoid in 2006 \cite{Neb06}.
A geodetic graph means a connected graph $G$ with a unique shortest $u$-$v$ path in $G$ for all $u,v\in V(G)$. 
In this paper, graphs have no multiple edges or loops.
\begin{definition}[travel groupoid]
Let $(V, \ast)$ be a groupoid, which is a pair of a non-empty set $V$ and a binary operation $\ast: V\times V \rightarrow V$ on $V$. A groupoid $(V, \ast)$ is called a \textit{travel groupoid} if it satisfies the following two conditions:\\
(t1) $(u\ast v)\ast u=u$ (for all $u, v \in V$);\\
(t2) If $(u\ast v)\ast v=u$, then $u=v$ (for all $u, v \in V$).
\end{definition}
See Proposition 1 for some relations of travel groupoids.
Let $(V,*)$ be a travel groupoid and $G=(V(G),E(G))$ a graph.
We say that $(V,*)$ is \textit{on} $G$ or that $G$ \textit{has} $(V,*)$ if $V(G)=V$ and
\[ E(G)=\{\{u,v\}\,|\,u, v \in V\mbox{ and }u\neq u\ast v=v \}. \]

For $u, v \in V$, we define $u\ast^0 v =u$ and $u\ast^{i+1}v = (u\ast^i v)* v$ for every $i\geq 0$.
For a travel groupoid $(V,*)$ on a graph $G$, $u,v\in V$, and $k\geq 1$,
the sequence 
\begin{equation} \label{intro-walk}
u*^0 v, \cdots, u*^{k-1}v, u*^{k}v
\end{equation}
is a walk in $G$.
This means the travel groupoid has information of the connections of vertices and the choice of walks.

The aim of this paper is to give an algorithm to construct a smooth travel groupoid.
\begin{definition}[smooth]
A travel groupoid $(V, \ast)$ is called \textit{smooth} if it satisfies the condition\\
(t4) if $u*v=u*w$, then $u*(w*v)=u*v$ (for any $u, v, w \in V$).\\
(We use the numberings of the conditions in [6].)
\end{definition}
%A smooth travel groupoid is an important travel groupoid, which is defined as a travel groupoid $(V,*)$ satisfying the following condition,
%\[ \mbox{if} \ u*v=u*w, \ \mbox{then}\ u*(w*v)=u*v \ \mbox{(for any}\  u, v, w \in V \mbox{)}.\]
If $(V,*)$ is a smooth travel groupoid, there is $k \geq 1$ such that $u*^k v= v$ and the sequence (\ref{intro-walk}) is a $u$-$v$ path in $G$.
See Proposition 3 and Proposition 5.

 In \cite{Neb06}, L.~Nebesk$\acute{\mbox{y}}$ proposed three questions about a travel groupoid.
The third question is recently solved in \cite{CPS2}.
Our construction provides an answer to the second question, which is as follows, in the finite case.
\begin{question} \label{q01}
 Does there exist a connected graph $G$ such that $G$ has no smooth travel groupoid? 
\end{question}

 The organization of this paper is as follows. 
In Section 2, we introduce properties of travel groupoids and some important results about travel groupoids.
In Section 3, we present an algorithm to construct smooth travel groupoids and exhibit it with some example.

\section{Definitions}
\par
In this section, we review definitions of some notations and results related to groupoids and graphs for later use.
\begin{definition}[simple]
A travel groupoid $(V, \ast)$ is called \textit{simple} if it satisfies the condition\\
(t3) if $v\ast u \neq u$, then $u\ast (v\ast u) =u\ast v$, for all $u, v \in V$.
\end{definition}
\begin{definition}[non-confusing]
Let $(V, *)$ be a travel groupoid, and take $u, v \in V$ such that $u\neq v$. If there exists $i\geq 3$ such that $u*^i v=u$, then we call the ordered pair $(u,v)$ a \textit{confusing pair} in $(V, *)$. If $(V, *)$ has no confusing pair, then we call it a \textit{non-confusing} travel groupoid.
\end{definition}
\par
We list some results for travel groupoids.
\begin{proposition}[\cite{Neb06}]  \label{prop01}
Let $(V, \ast)$ be a travel groupoid. Then the conditions (t1) and (t2) imply that\\
(1) $u\ast u = u$ for all $u\in V$,\\
(2) $u\ast v=v$ if and only if $v\ast u=u$ for all $u, v \in V$,\\
(3) $u\ast v=u$ if and only if $u=v$ for all $u, v \in V$ and\\
(4) $u\ast(u\ast v)= u\ast v$ for all $u, v \in V$.
\end{proposition}
\par
% Let $G$ be a graph. We say that a travel groupoid $(V, \ast)$ is \textit{on} $G$ or $G$ \textit{has} $(V, \ast)$ if $V(G)=V$ and $E(G)=\{\{u,v\}\,|\,u, v \in V\mbox{ and }u\neq u\ast v=v\}$. 
% A simple travel groupoid $(V,*)$ on a graph $G=(V, E)$ determine a unique path between every pair of vertices of $G$.
\begin{proposition}[\cite{Neb06}]
Let $(V, *)$ be a simple travel groupoid and let $k\geq 1$. For $x, y \in V$, if $x*^{k-1}y\neq y$ and $x*^k y=y$, then $y*^{k-1}x\neq x$ and $y*^j x= x*^{k-j}y$, where $0\leq j \leq k$.
\end{proposition}
\begin{proposition}[\cite{Neb06}]
Let $(V, \ast)$ be a finite travel groupoid on a graph $G$. 
Then $(V, \ast)$ is non-confusing if and only if the following statement holds for all distinct $u, v\in V$:
there exists $k \geq 1$ such that the sequence
\[
u*^0 v, \cdots, u*^{k-1}v, u*^{k}v
\]
is an $u$-$v$ path in $G$.
\end{proposition}
\begin{proposition}[\cite{Neb06}]
If $G$ has a non-confusing travel groupoid, then $G$ is connected.
\end{proposition}

\begin{proposition}[\cite{Neb06}] A smooth travel groupoid is non-confusing.
\end{proposition}

%We can easily check the following proposition from the definition of smoothness (see also Example \ref{exam01} below).
\begin{proposition}
Let $(V, *)$ be a travel groupoid. Then $(V, \ast)$ is smooth if and only if, for any $u, v \in V$, the set $V_{u, v}=\{w\in V\,|\, u*w=v \}$ is a subgroupoid of $(V, *)$.
\end{proposition}
\begin{proof}
Let $(V,\ast)$ be smooth. If $w, w' \in V_{u,v}$, then $u*w=u*w'=v$. 
From (t4), $u*(w*w')=u*w=u*w'=v$. Therefore $w*w'\in V_{u,v}$.
On the other hand, let $V_{u,v}$ be a subgroupoid of $(V,\ast)$.
If $u*w=u*w'=v$, then $w, w' \in V_{u,v}$ and $w*w' \in V_{u,v}$. 
%Now $V_{u,v}$ is a subgroupoid, then $w*w'\in V_{u,v}$. 
Thus $u*(w*w')=v$.
%\begin{flushright}
%\qed
%\end{flushright}
\end{proof}

The conditions have relationships shown in Figure \ref{rel}.

\begin{figure}[ht] 
$$\raisebox{0pt}{\begin{overpic}[height= 100 pt]{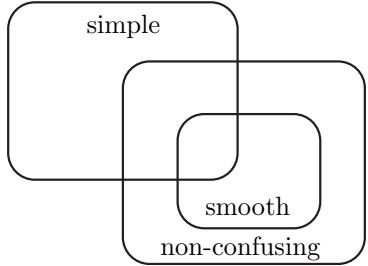}
\put(30,88){simple} \put(58,4){non-confusing} \put(75,19){smooth} 
\end{overpic}}$$
\caption{Relationships among the conditions.} \label{rel}
\end{figure}

\begin{theorem}[\cite{Neb06}]
For every finite connected graph $G$ there exists a simple non-confusing travel groupoid on $G$.
\end{theorem}

\begin{definition}[$v$-spanning tree]
Let $G$ be a graph and let $v$ be a vertex of $G$. We say that a spanning tree $T$ of $G$ is a \textit{$v$-spanning tree} if $T$ contains all incident edges of $v$. 
A $v$-spanning tree is called a \textit{$v$-tree} in \cite{CPS1}. 
\end{definition}
\par
In \cite{CPS1}, the number of non-confusing travel groupoids $(V, \ast)$ on a finite connected graph $G=(V, E)$ was determined by the product of the numbers of $v$-spanning trees of $G$ for all vertex $v\in V$.

\par
Let $G=(V, E)$ be a geodetic graph; a tree for example. 
For $u, v (\neq u) \in V$, there is a unique vertex $A_G(u,v)$ such that $d(u, A_G(u,v))=1$ and $d(A_G(u,v), v)=d(u,v)-1$, where $d(\,\cdot\,, \,\cdot\,)$ is the distance function. 
% There is a unique vertex $A_G(u, v)$ for $u, v (\neq u) \in V$ which satisfies that $d(u, A_G(u,v))=1$ and $d(A_G(u,v), v)=d(u,v)-1$ where $d(\,\cdot\,, \,\cdot\,)$ is the distance function. 

\section{Main result}
\par
In this section, we first present an algorithm which constructs a travel groupoid from a given finite connected graph.
Then we show an example of a travel groupoid on a graph according to the algorithm, which is smooth. 
Finally, as the main theorem, we prove that travel groupoids constructed by the algorithm are smooth and answer Question \ref{q01} for the finite graph case. 
\subsection{Algorithm} \label{subsec3-01}
\par
Let $G = (V, E)$ be a finite connected graph with no multiple edges and no loops. We construct a groupoid $(V,\ast)$ on $G$ and prove that $(V, \ast)$ is a travel groupoid.
\par
(Step 1)  Fix a vertex $o \in V$ and fix an $o$-spanning tree $T_o$ of $G$. We define a binary operation by $o\ast v = A_{T_o}(o, v)$ for any $v (\neq o) \in V$ and $o\ast o=o$.
\par
(Step 2) For any vertex $u (\neq o) \in V$, we construct a $u$-spanning tree $T_u$ from $T_o$.
Let $S(u)$ be the set of all the edges incident to $u$. 
Consider the set $L(u)=(E\setminus E(T_o)) \cap S(u)$ of edges, where $E(T_o)$ is the set of edges of $T_o$. 
Add an edge $e$ in $L(u)$ to $T_o$, then $T_o \cup \{e\}$ has a cycle. 
We then remove the edge on the cycle which is incident with the edge $e$ at the incident vertex of $e$ other than $u$.
By this, we obtain a new tree.
Now repeat this operation with the new tree in the place of $T_o$ for every edge in $L(u)$. 
Then, we will finally arrive at a $u$-spanning tree, and we define $T_u$ to be this tree.
\par
(Step 3)
We define a binary operation by $u\ast v = A_{T_u}(u, v)$ for any $v (\neq u) \in V$ and $u\ast u=u$. 

%Then we remove the incident edge of the opposite vertex of $e$ to $u$ which is in the cycle and have another tree. Repeating this operation for every edge in $L(u)$, we have a $u$-spanning tree and define $T_u$ by this tree.

\begin{remark}
For any vertex $u\, (\neq o) \in V$, the $u$-spanning tree $T_u$ derived from $T_o$ in the Step 2 is determined uniquely.
\end{remark}

\begin{proposition} Let $G=(V, E)$ be a finite connected graph. Then, groupoids constructed by the algorithm are travel groupoids on $G$.
\end{proposition}
\begin{proof}
Fix a vertex $o \in V$ and an $o$-spanning tree $T_o$. Then we construct a travel groupoid $(V, \ast)$ according to the algorithm. 
We check that $(V, \ast)$ satisfies the conditions (t1) and (t2).
\par
(t1) For any vertex $u, v\in V$, by the definition of $(V, \ast)$, $u*v$ is an adjacent vertex of $u$. So, the $u*v$-spanning tree $T_{u*v}$ contains the edge connecting $u$ and $u*v$. This implies $(u*v)*u=u$.
\par
(t2) Take $u, v \in V$ such that $u\neq v$. We show that $(u*v)*v\neq u$. Consider the unique $u$-$v$ path $P$ on $T_u$. $u*v$ is on $P$.  If $P$ is preserved in the change of spanning trees from $T_u$ to $T_{u*v}$ via $T_o$, $(u*v)*v=A_{T_{u*v}}(u*v, v)$ is on $P$ and $(u*v)*v\neq u$. We consider the case where $P$ is not preserved in the change of spanning trees from $T_u$ to $T_{u*v}$ via $T_o$. In the change from $T_u$ to $T_o$ the part of the path $P$ between $u*v$ and $v$ is preserved because the removed edges under the change are adjacent to $u$. 
Even though the edge between $u$ and $u*v$ was removed, this edge will be recovered after the change from $T_o$ to $T_{u*v}$ because $u$ is adjacent to $u*v$. 
Since $P$ is not contained in $T_{u*v}$, some edges on $P$ between $u*v$ and $v$ are not on $T_{u*v}$ and $P\cap T_{u*v}$ consists of at least two disjoint parts.
This implies that there are edges, outside of $P$, which connects $u*v$ to some vertices on $P$ beside $u$.
Then one of those vertices is $(u*v)*v$ because $v$ is on $P$ and one of the new edges connects $u*v$ to the part of $P$ on which $v$ lies. Hence, $(u*v)*v\neq u$.
\end{proof}

\begin{remark} 
For a fixed graph $G=(V,E)$, and $v$-spanning trees on $G$ for each vertex $v$, there are two ways to construct groupoids.
We call these ways a \textit{downward} construction and an \textit{upward} construction. 
The downward construction defines the binary operations of a groupoid $(V, \ast)$ by $u\ast v = A_{T_v}(u, v)$ for any $u (\neq v) \in V$ and $v\ast v = v$.
The upward construction defines the binary operations of a groupoid $(V, \ast')$ by $u\ast'v = A_{T_u}(u, v)$ for any $u (\neq v) \in V$ and $v\ast'v = v$.
The downward construction is used in \cite{CPS1},  and we use the upward construction in the algorithm in the subsection \ref{subsec3-01}.
%$v$-spanning trees $T_v$ of a graph $G=(V, E)$ are used for defining the binary operations of a groupoid $(V, \ast)$ by $u\ast v = A_{T_v}(u, v)$ for any $u (\neq v) \in V$ and $v\ast v = v$, which is the opposite direction to our case. 
%We call this way of construction of groupoids on $G$ a \textit{downward} construction and our way an \textit{upward} construction. 
A groupoid constructed by the downward construction is always a travel groupoid (\cite{CPS1}). 
On the other hand, in general, a groupoid constructed by the upward construction without using our algorithm may not be a travel groupoid (see Example \ref{exam04} below).
\end{remark}

\subsection{Example} \label{subsec3-02}
\par
We show an example of a smooth travel groupoid $(V, \ast)$ on a finite connected graph $G_1 = (V, E)$ in Figure \ref{fig01} constructed by the algorithm in the subsection \ref{subsec3-01}.

\begin{figure}[ht] 
$$\raisebox{0pt}{\begin{overpic}[height= 65 pt]{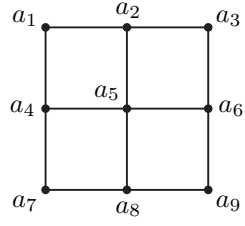}
\put(-11,65){$a_1$} \put(28,68){$a_2$} \put(66,65){$a_3$}
\put(-12,31){$a_4$} \put(20,37){$a_5$} \put(67,31){$a_6$}
\put(-11,-5){$a_7$} \put(28,-7){$a_8$} \put(66,-5){$a_9$}
\end{overpic}}$$
\caption{A graph $G_1$.} \label{fig01}
\end{figure}

\begin{example} \label{exam01}
For every vertex $v\in V$ of $G_1$, we prepare a $v$-spanning tree. We first fix a vertex $a_1$ as $o$ in the algorithm and an $a_1$-spanning tree and then construct $a_i$-spanning trees for other vertices $a_i$ following the algorithm as in Figure \ref{exm02}.
\begin{figure}[ht] 
$$\raisebox{0 pt}{\begin{overpic}[height= 65 pt]{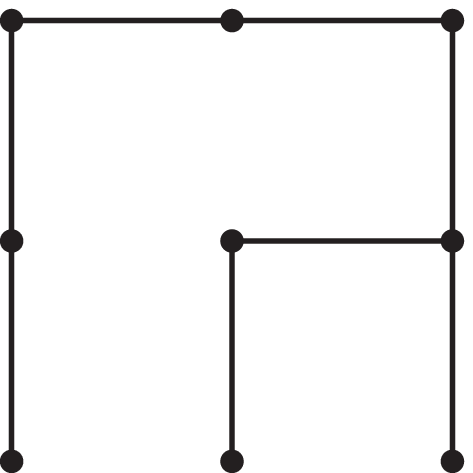}
\put(-11,65){$a_1$} \put(66,65){$a_3$} \put(67,31){$a_6$}
\end{overpic}}
\hspace{1.5cm}
\raisebox{0 pt}{\begin{overpic}[height= 65 pt]{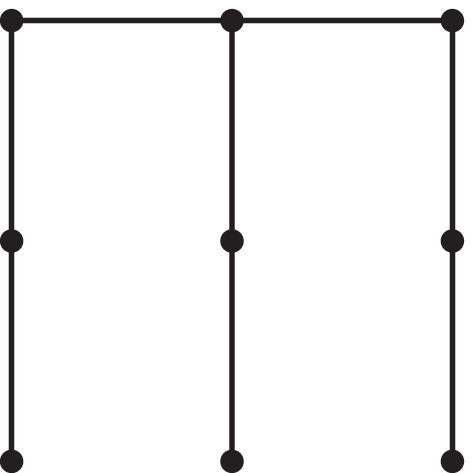}
\put(28,68){$a_2$}
\end{overpic}}
\hspace{1.5cm}
\raisebox{0 pt}{\begin{overpic}[height= 65 pt]{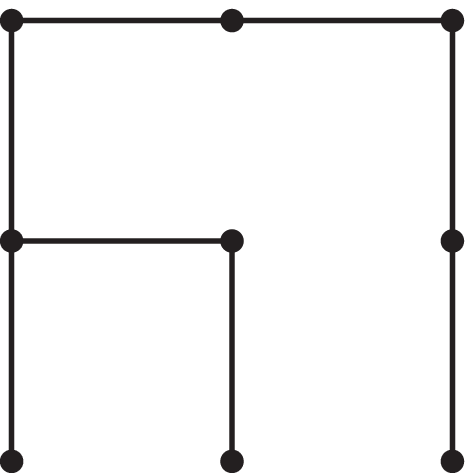}
\put(-12,31){$a_4$}
\end{overpic}}
$$\\
$$\raisebox{0 pt}{\begin{overpic}[height= 65 pt]{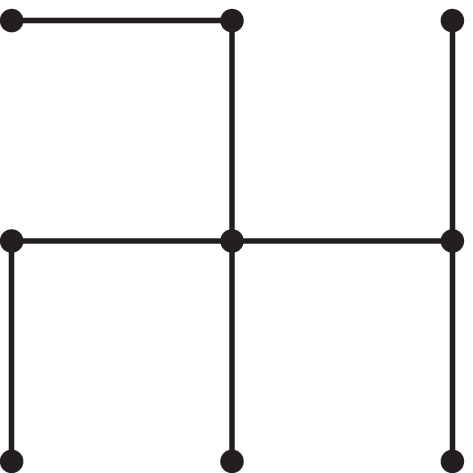}
\put(20,37){$a_5$}
\end{overpic}}
\hspace{1.5cm}
\raisebox{0 pt}{\begin{overpic}[height= 65 pt]{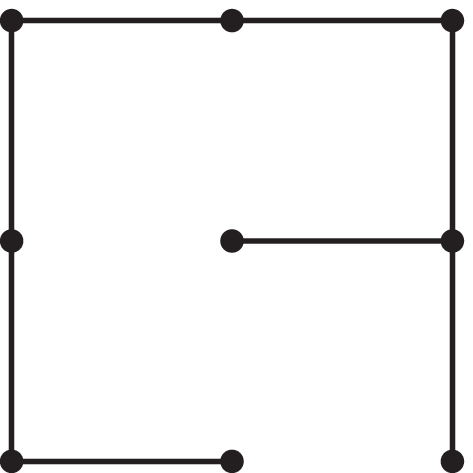}
\put(-11,-5){$a_7$}
\end{overpic}}
$$\\
$$\raisebox{0 pt}{\begin{overpic}[height= 65 pt]{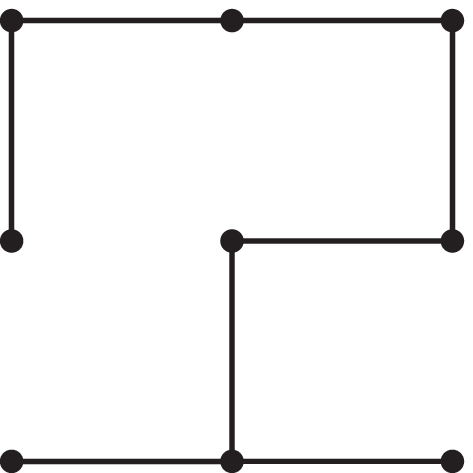}
\put(28,-7){$a_8$}
\end{overpic}}
\hspace{1.5cm}
\raisebox{0 pt}{\begin{overpic}[height= 65 pt]{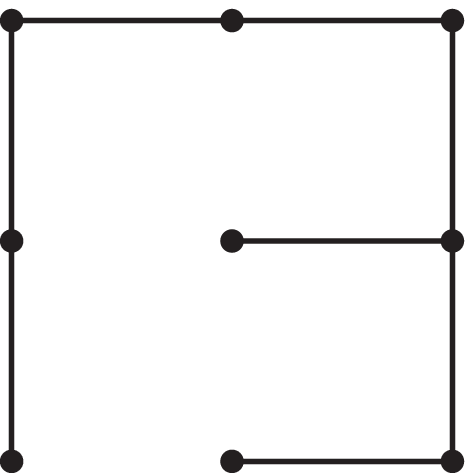}
\put(66,-5){$a_9$}
\end{overpic}}$$
\caption{$v$-spanning trees for $(V, \ast)$.} \label{exm02}
\end{figure}
Here, a tree graph represents an $a_i$-spanning tree if the figure includes the letter of the vertex $a_i$ i.e. $a_1$, $a_3$ and $a_6$ have the same $v$-spanning tree. The travel groupoid $(V, \ast)$ constructed from the spanning trees has the binary operation in Table \ref{table}.
\begin{table}[ht] 
\centering
\caption{Table of binary operations of $(V, \ast)$.} \label{table}
\begin{tabular}{c|c c c c c c c c c }
$*$ & $a_1$ & $a_2$ & $a_3$ & $a_4$ & $a_5$ & $a_6$ & $a_7$ & $a_8$ & $a_9$\\
\hline
$a_1$ & $a_1$ & $a_2$ & $a_2$ & $a_4$ & $a_2$ & $a_2$ & $a_4$ & $a_2$ & $a_2$\\
$a_2$ & $a_1$ & $a_2$ & $a_3$ & $a_1$ & $a_5$ & $a_3$ & $a_1$ & $a_5$ & $a_3$\\
$a_3$ & $a_2$ & $a_2$ & $a_3$ & $a_2$ & $a_6$ & $a_6$ & $a_2$ & $a_6$ & $a_6$\\
$a_4$ & $a_1$ & $a_1$ & $a_1$ & $a_4$ & $a_5$ & $a_1$ & $a_7$ & $a_5$ & $a_1$\\
$a_5$ & $a_2$ & $a_2$ & $a_6$ & $a_4$ & $a_5$ & $a_6$ & $a_4$ & $a_8$ & $a_6$\\
$a_6$ & $a_3$ & $a_3$ & $a_3$ & $a_3$ & $a_5$ & $a_6$ & $a_3$ & $a_5$ & $a_9$\\
$a_7$ & $a_4$ & $a_4$ & $a_4$ & $a_4$ & $a_4$ & $a_4$ & $a_7$ & $a_8$ & $a_4$\\
$a_8$ & $a_5$ & $a_5$ & $a_5$ & $a_5$ & $a_5$ & $a_5$ & $a_7$ & $a_8$ & $a_9$\\
$a_9$ & $a_6$ & $a_6$ & $a_6$ & $a_6$ & $a_6$ & $a_6$ & $a_6$ & $a_8$ & $a_9$\\
\end{tabular}
\end{table}
We can check that for any $u,v,w\in V$, if $u*v=u*w$, then $u*(v*w)=u*v$. For example, if we put $u=a_1$, then for $i=2, 3, 5, 6, 8, 9$, $a_1*a_i=a_2$. We focus on the intersection of $a_i$-rows and $a_j$-columns of the table, where $i, j=2, 3, 5, 6, 8, 9$. In the intersection, the results of products are only $a_i$ $(i=2, 3, 5, 6, 8, 9)$ i.e. $\{a_2, a_3, a_5, a_6, a_8, a_9\}$ is a subgroupoid of $(V, \ast)$. This means that $a_1*(a_i*a_j)=a_2=a_1*a_i$ for $i,j=2, 3, 5, 6, 8, 9$. The similar things hold for other cases and we see that $(V, *)$ is smooth.
\end{example}

\subsection{Main theorem} \label{subsec3-03}
\begin{theorem}
Let $G=(V, E)$ be a finite connected graph. Then a travel groupoid on $G$ constructed by the algorithm in Subsection \ref{subsec3-01} is smooth.
\end{theorem}
\begin{proof}
Fix a vertex $o \in V$ and an $o$-spanning tree $T_o$. Then construct the travel groupoid $(V, \ast)$ by the algorithm. 
We need only to check the condition (t4): if $u*v=u*w$, then $u*(w*v)=u*v$ for any $u, v, w \in V$. A \textit{branch} of a $u$-spanning tree $T_u$ is a component of the graph $(V(T_u)\setminus u, E(T_u)\setminus S(u))$ where $V(T_u)$ is the set of vertices of $T_u$ and $E(T_u)$ is the set of edges of $T_u$.
\par
(i) We show that if $o*v=o*w$, then $o*(v*w)=o*v$ for any $v, w \in V$. From Proposition \ref{prop01} (3),  this is obvious if $v=o$. We assume $v\neq o$. Since $o*v=o*w$, $v$ and $w$ are on the same branch $B$ of $T_o$. Consider the unique $v$-$w$ path $P$ on $T_o$, which is on $B$. If $P$ is preserved in the change of the spanning trees from $T_o$ to $T_v$, $v\ast w$ is on $P$ and on the branch $B$. Hence $o*(v*w)=o*v$. 
If $P$ is not preserved in the change of the spanning trees from $T_o$ to $T_v$, some edges in $P$ are removed and some edges of $T_v$ connect $v$ and vertices on $P$. This implies that $v\ast w$ is on $P$ and so on $B$. Hence $o*(v*w)=o*v$.
\par
(ii) We show that if $u*v=u*w$, then $u*(v*w)=u*v$ for any $u(\neq o), v, w \in V$. From Proposition \ref{prop01} (3), this is obvious if $v=u$. We assume $v\neq u$. Since $u*v=u*w$, $v$ and $w$ are on the same branch $B$ of $T_u$. Consider the unique $v$-$w$ path $P$ on $T_u$, which is on $B$. We consider the change of the spanning tree from $T_u$ to $T_o$. In the change $P$ is preserved because the edges of $T_u$ which are not of $T_o$ are all incident to $u$ and so not on $B$.
\end{proof}

\par
As a corollary to this theorem, we have an answer to Question \ref{q01} for finite graphs.
\begin{theorem}
For any finite connected graph $G$, there exists a smooth travel groupoid on $G$.
\end{theorem}

\subsection{Remarks} In this subsection, we show two examples of groupoids. 
One is a groupoid constructed from the graph $G_2$ drawn in Fig. \ref{fig04} by the upward construction using $v$-spanning trees of $G_2$ (but not following the algorithm in the subsection \ref{subsec3-01}) which is not a travel groupoid. 
This example means that a groupoid constructed by upward construction is not necessarily a travel groupoid on a graph. (In contrast, a groupoid constructed by the downward construction is always a travel groupoid \cite{CPS1}.) Another example is about a travel groupoid on the graph $G_3$ drawn in Fig. \ref{fig05} constructed by the algorithm in the subsection \ref{subsec3-01} which is not simple.

\begin{example}\label{exam04} Let $G_2=(V, E)$ be a graph depicted in the first figure in Figure \ref{fig04}. We construct $v$-spanning trees of vertices $v$ of $G_2$ as the other graphs in Figure \ref{fig04}. Here, a tree graph represents a $v$-spanning tree if the figure includes the letter of the vertex $v$. The groupoid $(V, \ast)$ defined by the trees upwardly is not a travel groupoid because $(a_2*a_5)*a_5=a_3*a_5=a_2\neq a_5$ yielding that $(V, \ast)$ does not satisfy (t2).
\end{example}
\begin{figure}[ht]  
$$\raisebox{0 pt}{\begin{overpic}[height= 60 pt]{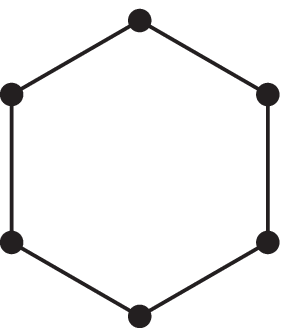}
\put(23,62){$a_1$} \put(54,44){$a_2$}
\put(54,12){$a_3$} \put(23,-7){$a_4$}
\put(-12,12){$a_5$} \put(-12,44){$a_6$}
\end{overpic}}
\hspace{1.7cm}
\raisebox{0 pt}{\begin{overpic}[height= 60 pt]{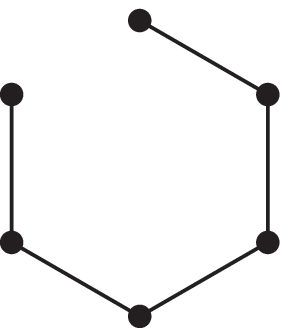}
 \put(54,44){$a_2$}
\put(23,-7){$a_4$}
\put(-12,12){$a_5$}
\end{overpic}}
\hspace{1.7cm}
\raisebox{0 pt}{\begin{overpic}[height= 60 pt]{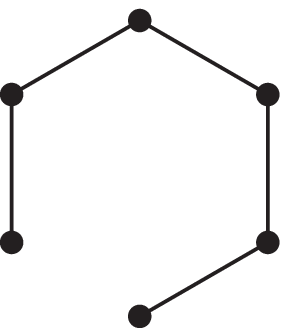}
\put(23,62){$a_1$} \put(54,12){$a_3$} 
\put(-12,44){$a_6$}
\end{overpic}}$$
\caption{A graph $G_2$ and its $v$-spanning trees.} \label{fig04}
\end{figure}
\par
There is a smooth travel groupoid constructed by the algorithm in the subsection \ref{subsec3-01} which is not simple.
\begin{example} Let $G_3=(V, E)$ be a graph depicted in the first figure in Figure \ref{fig05}. 
We construct the travel groupoid $(V,\ast)$ on $G_3$ from $v$-spanning trees of vertices $v$ of $G_3$ by the algorithm in the subsection \ref{subsec3-01}.
The $v$-spanning trees are in Figure \ref{fig05}.
 %as the other graphs in Figure \ref{fig05} by the algorithm in the subsection \ref{subsec3-01}. 
This travel groupoid is not simple because we can easily check that $a_4*a_2\neq a_2$ and $a_2*a_4=a_3\neq a_1 = a_2*(a_4*a_2)$ yielding that $(V, \ast)$ does not satisfy (t3).
\end{example}

\begin{figure}[ht] 
$$\raisebox{0 pt}{\begin{overpic}[height= 70 pt]{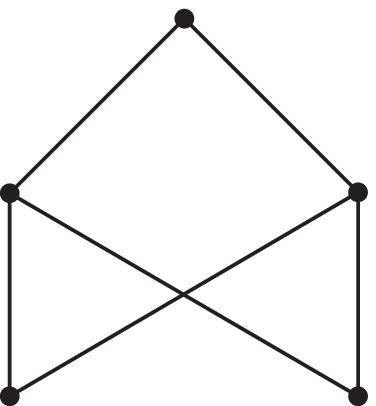}
\put(30,73){$o$}
\put(-12,39){$a_1$} \put(-12,-3){$a_2$}
\put(67,39){$a_3$} \put(67,-3){$a_4$}
\end{overpic}}
\hspace{1.5cm}
\raisebox{0 pt}{\begin{overpic}[height= 70 pt]{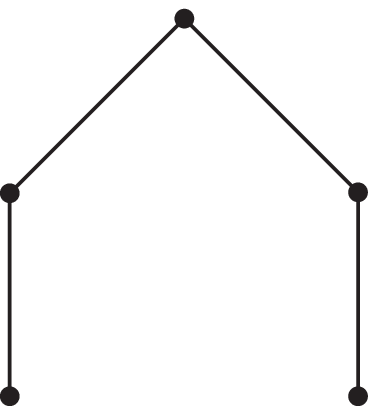}
\put(30,73){$o$}
\end{overpic}}
\hspace{1.5cm}
\raisebox{0 pt}{\begin{overpic}[height= 70 pt]{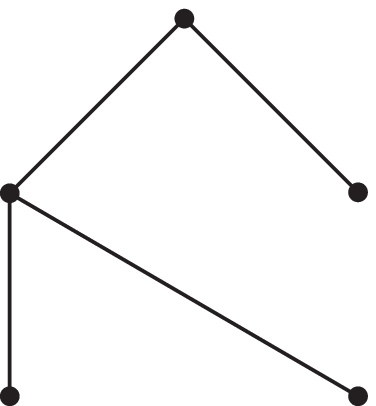}
\put(-12,39){$a_1$}
\end{overpic}}$$
$$\raisebox{0 pt}{\begin{overpic}[height= 70 pt]{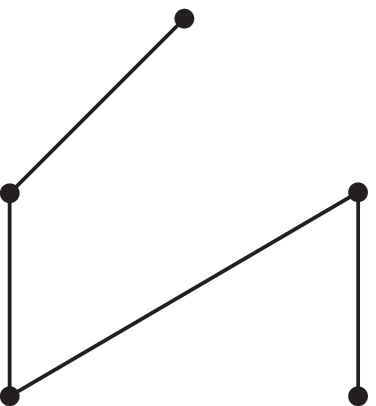}
\put(-12,-3){$a_2$}
\end{overpic}}
\hspace{1.5cm}
\raisebox{0 pt}{\begin{overpic}[height= 70 pt]{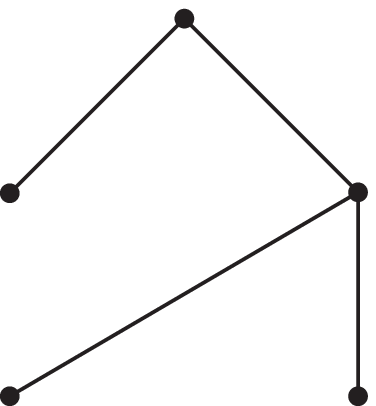}
\put(67,39){$a_3$}
\end{overpic}}
\hspace{1.5cm}
\raisebox{0 pt}{\begin{overpic}[height= 70 pt]{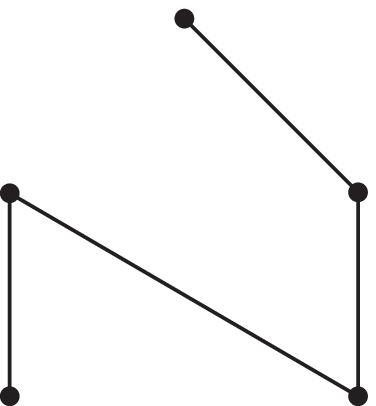}
\put(67,-3){$a_4$}
\end{overpic}}$$
\caption{A graph $G_3$ and its $v$-spanning trees.} \label{fig05}
\end{figure}

\end{document}